\documentclass[11pt]{article}
\usepackage{amssymb,amsfonts,amsmath,latexsym,epsf,tikz,url}
\usepackage[usenames,dvipsnames]{pstricks}
\usepackage{pstricks-add}
\usepackage{epsfig}
\usepackage{pst-grad} % For gradients
\usepackage{pst-plot} % For axes
\usepackage[space]{grffile} % For spaces in paths
\usepackage{etoolbox} % For spaces in paths
\makeatletter % For spaces in paths
\patchcmd\Gread@eps{\@inputcheck#1 }{\@inputcheck"#1"\relax}{}{}
\makeatother

\newtheorem{theorem}{Theorem}[section]
\newtheorem{proposition}[theorem]{Proposition}
\newtheorem{observation}[theorem]{Observation}

\newtheorem{corollary}[theorem]{Corollary}

\newtheorem{remark}[theorem]{Remark}

\newcommand{\qed}{\hfill $\square$\medskip}

\begin{document}

\title{ Stability of $2$-domination number of a graph }

\author{ M. Mehraban  and S. Alikhani$^{}$\footnote{Corresponding author} 
}

\date{\today}

\maketitle

\begin{center}
	
Department of Mathematical Sciences, Yazd University, 89195-741, Yazd, Iran

	\bigskip
	{\tt  Mazharmehraban2020@gmail.com, ~~alikhani@yazd.ac.ir  
		}

\end{center}

%%%%%%%%%%%%%%ABSTRACT%%%%%%%%%%%%%%%%%%%%%%%%%%%%%%%%%%%%%%%%%%%%%%%%%%%%%%%%%%%%
 	\begin{abstract}
 		This paper delves into the stability of the $2$-domination number in simple undirected graphs. The $2$-domination number of a graph $G$, $\gamma_2(G)$, represents the minimum size of a vertex subset where every other vertex in the graph is adjacent to at least two members of the subset. We define the $2$-domination stability, $st_{\gamma_2}(G)$, as the smallest number of vertices whose removal causes a change in $\gamma_2(G)$. Our primary contributions include computing this parameter for specific graphs, establishing various bounds for this stability and determining its behavior under certain graph operations combining two graphs.
 		 	\end{abstract}
 

\medskip
\noindent{\bf AMS Subj.\ Class.}:  05C05, 05C69.
%%%%%%%%%%%%%%%%%%%%%%%%%%%%%%%%%%%%%%%%%%%%%%%%%%%%%%%%%%%%%%%%%%%%%%%%%%%%%%%%%
%%%%%%%%%%%%%%%%%%%%%%%%%%%%%%%%%%%%%%%%%%%%%%%%%%%%%%%%%%%%%%%%%%%%%%%%%%%%%%%%%
\section{Introduction}
 Let $G=(V,E)$ be a simple graph with finite number of vertices. 
  The open neighborhood of a vertex $v\in V(G)$  is the set of vertices that are adjacent to $v$, but not including $v$ itself, $N(v) = \{u \in V(G) : uv \in E(G)\}$ and the closed neighborhood of a vertex $v\in V(G)$ is the open neighborhood of $v$ along with the vertex $v$ itself and is denoted as $N[v] = N(v) \cup \{v\}.$
   For a set $S\subseteq V(G)$, the open neighborhood of $S$ is $N(S)=\bigcup_{v\in S} N(v)$ and  the closed neighborhood of $S$
   is $N[S]=N(S)\cup S$. The private neighborhood $pn(v,S)$ of $v\in S$ is defined by $pn(v,S)=N(v)-N(S-\{v\})$, equivalently, $pn(v,S)=\{u\in V| N(u)\cap S=\{v\}\}$.
   The degree of a vertex $v$  denoted as $deg(v)$, is the number of edges incident to that vertex which is equal to $|N(v)|$.  A leaf of a tree is a vertex of degree 1.  A subset \( D \subseteq V(G) \) is called a dominating set of \( G \) if every vertex \( v \in V \setminus D \) has at least one neighbor in \( D \); that is,
   $N(v) \cap D \neq \emptyset \quad \text{for all } v \in V \setminus D$.
   The domination number of \( G \), denoted by \( \gamma(G) \), is the minimum cardinality of a dominating set of \( G \) (see \cite{8,9}).
   A domination-critical vertex in a graph $G$ is a vertex whose
   removal decreases the domination number. It is easy to observe that for any graph $G$ we have $\gamma(G)-1\leq\gamma(G+e)\leq\gamma(G)$ for every edge $e\notin E(G)$.
    Sumner and Blitch in \cite{14} have defined domination critical graphs. A graph G is said to be domination
   critical, or $\gamma$-critical, if $\gamma(G+e)=\gamma(G)-1$ for every edge $e$ in the complement $G^c$ of
   $G$. A graph is said to be domination stable, or $\gamma$-stable, if $\gamma(G)=\gamma(G+e)$ for every
   edge $e$ in the complement $G^c$ of $G$. For detailed information and results regarding the concept of domination critical graphs, we refer interested readers to the papers \cite{5, 14, 15}. Bauer et. al introduced the concept of domination stability in graphs in $1983$ \cite{3}. After then, was studied by Rad, Sharifi and Krzywkowski in \cite{11}. Stability for different types of domination parameters has been investigated in the literature, for example, in \cite{2, 6, 10, 12, wcds}. This subject has considered and studied for another
   graph parameters. For example see \cite{saeid, saeid2,4}.
   
	A subset \( D \subseteq V(G) \) is called a 2-dominating set of the graph \( G \) if every vertex \( v \in V \setminus D \) has at least two neighbors in \( D \); that is,
		$|N(v) \cap D| \geq 2 \quad \text{for all } v \in V \setminus D$.
		The 2-domination number of \( G \), denoted by \( \gamma_2(G) \), is the minimum cardinality of a 2-dominating set in \( G \):
		$\gamma_2(G) = \min \left\{ |D| \,\middle|\, D \subseteq V(G), \forall v \in V \setminus D,\, |N(v) \cap D| \geq 2 \right\}$. For more details see \cite{M1}.
 		We define the stability of $2$-domination number of the graph \( G \),  \( st_{\gamma_2}(G) \),  as the minimum number of vertices that must be removed from \( G \) in order to change its $2$-domination number. 
 		
	\medskip
 	In the next section, we compute the value of stability of $2$-domination number for some special classes of graphs. We find some bounds on the stability of $2$-domination number in Section 3. The stability of $2$-domination number of some operations of two graphs is studied in Section 4.  Finally, we conclude
 	the paper in Section 5. 
 	
\section{Stability of 2-domination number of certain  graphs}

	In this section, we compute the stability of  $2$-domination number  of the graph \( G \) for some specific graphs. 
	
	\subsection{Results for specific graphs} 
	
	We start with the following observation:
	
\begin{observation}{\rm \cite{M1}} 
	 If $P_n$ and $C_n$ are the Path graph and the cycle graph of order $n\geq 4$, then 
	\begin{enumerate} 
		 	\item[(i)] 
	 		 \[\gamma_2(P_n)=
		\begin{cases}
		\frac{n}{2} + 1, & \text{if } n \text{ is even}, \\
		\frac{n - 1}{2} + 1, & \text{if } n \text{ is odd}.
		\end{cases}
		\]
		\item[(ii)] 
				 \[
		\gamma_2(C_n) =
		\begin{cases}
		\frac{n}{2}, & \text{if } n \text{ is even}, \\
		\frac{n+1}{2}, & \text{if } n \text{ is odd}.
		\end{cases}
		\]
	\end{enumerate} 
\end{observation}
We obtain the stability of $2$-domination number of some specific graphs. 
\begin{proposition}
	\begin{enumerate} 
		\item[(i)] 
		For $n\geq 4$, $st_{\gamma_2}(P_n)=3.$
	 \item[(ii)]
	 	$st_{\gamma_2}(C_n)=\left\{
	 	\begin{array}{cc}
	 	2    &\quad if\; n\ is\; odd,\\
	 	3    &\quad if\; n\ is\;  even.\\
	 	\end{array}\right.$
	 	
	 	\item[(iii)] 
	 		If $W_n$ is a wheel graph (join of $K_1$ and $C_{n-1}$, i.e., $K_1\vee C_{n-1}$), then for $n\geq 5$, $st_{\gamma_2}(W_n) = 2$.  
	\end{enumerate}
	\end{proposition}
\begin{proof} 
	\begin{enumerate} 
		\item[(i)] 
		By removing the three first consecutive vertices of $P_n$, the $2$-domination number of $P_n$ will be changes but by removing two vertices, this parameter does not change. So we have the result.

	\item[(ii)] 
	Suppose that $V(C_n)=\{v_1,v_2,\dots,v_n\}$. For odd $n$, by removing
	 two consecutive vertices $\{v_1,v_2\}$, we will have $P_{n-2}$ which  its $2$-domination number is one less than the $2$-domination number of $C_n$.
	 For even $n$, we need to remove three consecutive vertices $\{v_1,v_2,v_3\}.$   
	 \item[(iii)] 
	 		Let $W_n=K_1\vee C_{n-1}$. By removing two adjacent vertices of $C_{n-1}$ the $2$-domination number of $W_n$ which is 	$\lfloor \frac{n+1}{3} \rfloor + 1$ (see \cite{M1}) will be changed. \qed
	 		\end{enumerate} 
\end{proof}

Now we obtain the stability of $2$-domination number of friendship graph and book graph. The friendship graph $F_n$ is a graph that can be constructed by coalescence $n$ copies of the cycle graph $C_3$ of length $3$ with a common vertex. The friendship graph $F_n$ is a graph with the property that every two vertices have exactly one neighbor in common are exactly the friendship graphs \cite{erdos}.
The $n$-book graph $(n\geq2)$ is defined as the Cartesian product $K_{1,n}\square P_2$. We call every $C_4$ in the book graph $B_n$, a page of $B_n$. All pages in $B_n$ have a common side $v_1v_2$.   Figure \ref{friend} shows the  friendship graph $F_n$ and the book graph $B_n$.

  \begin{figure}[ht]
  	\hspace{1cm}
  	\begin{minipage}{6.3cm}
  		\includegraphics[width=\textwidth]{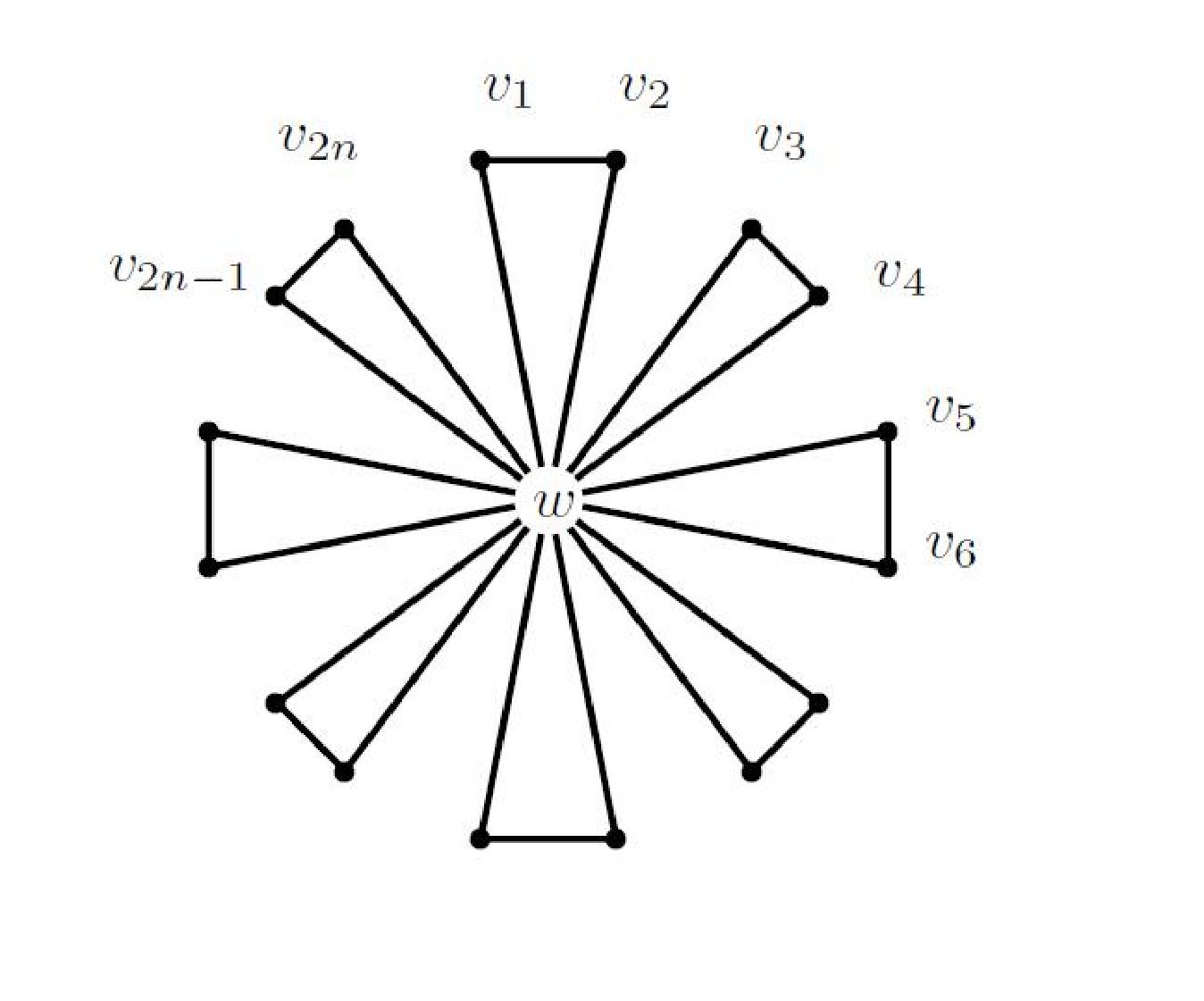}
  	\end{minipage}
  	\begin{minipage}{6.1cm}
  		\includegraphics[width=\textwidth]{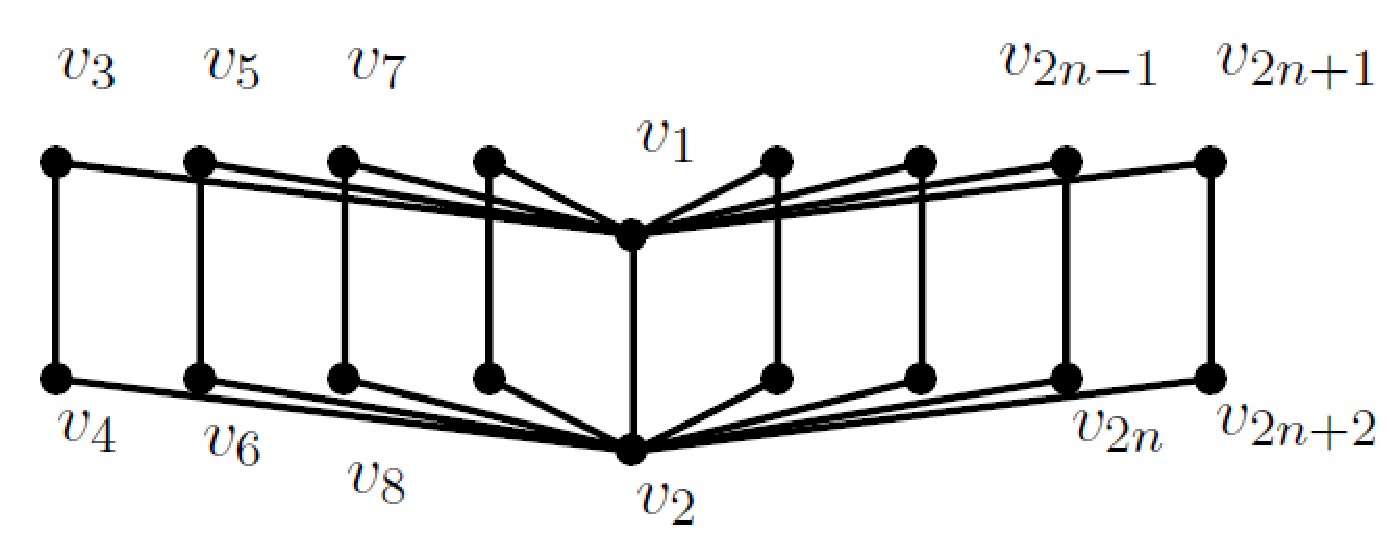}
  	\end{minipage}
  	\caption{\label{friend} Friendship graph $F_n$ and book graph $B_n$,  respectively.}
  \end{figure}
  
  The following observation gives the $2$-domination number of the friendship graph and  the book graph.
  
\begin{observation} {\rm \cite{M1}} 
	\begin{enumerate}
	\item[(i)] For $n\geq 1$, $\gamma_2(F_n)=n+1$.
		
		\item[(ii)]
		 For $n\geq 2$, $\gamma_2(B_n)=n+1$. 
		\end{enumerate} 
\end{observation}
The following proposition gives the stability of $2$-domination number of $F_n$ and $B_n$. 

\begin{proposition}
	\begin{enumerate}
		\item[(i)]
For any $n\geq 2$, $st_{\gamma_2}(F_n)=1$.  
	\item[(ii)] 
	For any $n$,  $st_{\gamma_2}(B_n)=1$. 
		\end{enumerate} 
\end{proposition}
\begin{proof}
		\begin{enumerate}
			\item[(i)] By removing the central vertex (the vertex $w$ in Figure \ref{friend}) the $2$-domination number change. 
			
			\item[(ii)] We have the result by removing two vertices $v_1,v_2$ of $B_n$ (Figure \ref{friend}). \qed
\end{enumerate} 
\end{proof}

\begin{observation} {\rm \cite{M1}} 
	\begin{enumerate}
	\item[(i)]	If $K_n$ is a complete graph for every  $n\geq 2$, $\gamma_2(K_n)=2$.
		
  	\item[(ii)] For $n\geq 3$, $\gamma_2(K_{1,n})=n-1$.
		
		\item[(iii)] $\gamma_2(K_{m,n}) = 
	\begin{cases}
	n - 1, & \text{if } m = 1, \\
	m - 1, & \text{if } n = 1, \\
	4, & \text{if } m, n \geq 2.
	\end{cases}
	$
	\end{enumerate} 
\end{observation}

Now, we state the value of the stability of $2$-domination number of $K_n$, $K_{1,n}$ and $K_{m,n}$ which are easy to obtain.
\begin{proposition}
	\begin{enumerate}
		\item[(i)]
		For every  $n\geq 2$, $st_{\gamma_2}(K_n)=n-1.$
		\item[(ii)] 
		For $n\geq 2$, $st_{\gamma_2}(K_{1,n})=1.$
		\item[(iii)] 
		$st_{\gamma_2}(K_{m,n}) = 
		\begin{cases}
		1, & \text{if } m = 1, \; or\; n=1\,\\
		2, & \text{if } m,n \geq 2.
		\end{cases}$
		\end{enumerate} 
\end{proposition}

\subsection{Results for cactus graphs}

In this subsection, we obtain the stability of $2$-domination number of cactus graphs. 
A cactus graph is a connected graph where each edge belongs to at most one cycle. Therefore, every block of a cactus graph is either a single edge or a cycle. When all blocks in a cactus graph $G$ are cycles of identical length $m$, the graph is called an $m$-uniform cactus.

A triangular cactus is a connected graph where each block is a triangle, making it a 3-uniform cactus. A vertex that belongs to multiple triangles is called a cut-vertex. When every triangle contains at most two cut-vertices, and each cut-vertex is shared by exactly two triangles, the graph forms a chain triangular cactus. The length of this chain is the number of triangles it contains. Such a structure, denoted by $T_n$, has $2n + 1$ vertices and $3n$ edges (\cite{M2}) (see Figure~\ref{cactus}).
By extending this idea to cycles of length four, we define square cacti, where each block is a $C_4$. Chain square cacti, denoted by $Q_n$, vary depending on how internal squares connect to each other (see Figure \ref{cactus}). If the cut-vertices of a square are adjacent, it is called an {ortho-square}; otherwise, it is a {para-square}. The chain consisting solely of para-squares is denoted by $Q_n$, while the chain formed by ortho-squares is denoted by $O_n$ (illustrated in Figure~\ref{cactus}).

 \begin{figure}[ht]
 	\hspace{1cm}
 	\begin{minipage}{4.4cm}
 		\includegraphics[width=\textwidth]{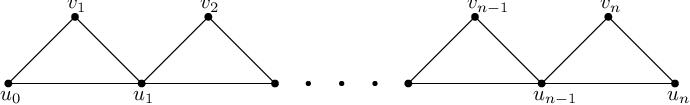}
 	\end{minipage}
 	\begin{minipage}{4.4cm}
 		\includegraphics[width=\textwidth]{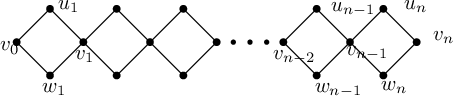}
 	\end{minipage}
 	\begin{minipage}{4.4cm}
 		\includegraphics[width=\textwidth]{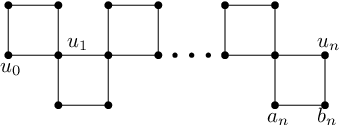}
 	\end{minipage}
 	\caption{\label{cactus} The cactus  $T_n$, $Q_n$ and $O_n$,  respectively.}
 \end{figure}

\begin{observation}
	\begin{enumerate}
\item[(i)] 
 $\gamma_2(T_n)=\gamma_2(Q_n)=\lceil\frac{n+2}{2}\rceil$.
 \item[(ii)] 
 For the ortho-chain square cactus graph, $\gamma_2(O_n)=n+1$.
 \end{enumerate} 
\end{observation}

The following theorem gives the stability of $2$-domination number of these three cactus which is easy to obtain. 

\begin{proposition}
For $n\geq 2$,		$st_{\gamma_2}(T_n)=st_{\gamma_2}(Q_n)=st_{\gamma_2}(O_n)=2$.
		\end{proposition}

We close this section by the following theorem: 

\begin{theorem}
	\begin{enumerate}
		\item[(i)]
		There exist graphs $G$ and $H$ with the same $2$-domination stability such that $|\gamma_2(G)-\gamma_2(H)|$ is arbitrarily large.
		\item[(ii)]
		There exist graphs $G$ and $H$ with the same $2$-domination number such that $|st_{\gamma_2}(G)-st_{\gamma_2}(H)|$ is arbitrarily large.
			\end{enumerate} 		
\end{theorem} 
\begin{proof}
	\begin{enumerate}
		\item[(i)] 
		Suppose that $G$ is the helm graph which is obtained from a wheel graph of order $n$ (i.e., $W_n=K_1\vee C_{n-1}$) by appending a single
		pendant edge to each vertex of cycle graph. We know that for $n\geq 3$, 
	$\gamma_2(G) = \lfloor \frac{n}{2} \rfloor + 1$ (see \cite{M1}).  If we remove the central vertex from $G$, say $v$, then 
	\[
	\gamma_2(G-\{v\})= n \neq \lfloor \frac{n}{2} \rfloor + 1,
	\]
	and so $st_{\gamma_2}(G)=1$. 
	
	Now consider the graph $H$ which is obtained from a wheel graph of order $n$ (i.e., $W_n=K_1\vee C_{n-1}$) by inserting a new vertex between any pair of adjacent vertices on the cycle $C_{n-1}$. It is easy to see that 
	for $n\geq 3$, $\gamma_2(H)=n$ (see \cite{M1}). By removing one vertex from $C_{n-1}$, say $v$, we have 
\[
\gamma_2(G-\{v\})= n -1\neq n,
\]
and so $st_{\gamma_2}(H)=1$. So we have the result.  
\item[(ii)] 
Consider the graph $G=P_4$ and $H=K_n$. We have $\gamma_2(P_4)=\gamma_2(K_n)=2$. On the other hand, $st_{\gamma_2}(P_n)=3$ and $st_{\gamma_2}(K_n)=n-1$ and so
we have the result.\qed
	
\end{enumerate}
\end{proof}

\section{Bounds on $st_{\gamma_2}(G)$}
In this section, we derive some bounds related to the stability of $2$-domination number in graphs.
First, we study the relationship between the stabilities of $2$-domination number of graph $G$ and $G-v$, where $v\in  V (G)$. Also we obtain upper bounds for $st_{\gamma_2}$(G).
\begin{proposition}
	Let $G$ be a graph and $v$ be a vertex of $G$. Then
	\[
	st_{\gamma_2}(G) \leq st_{\gamma_2}(G - v) + 1.
	\]
\end{proposition}
\begin{proof}
	If $\gamma_2(G) = \gamma_2(G - v)$, then we have
	$st_{\gamma_2}(G) \leq st_{\gamma_2}(G - v) + 1.$
	If $\gamma_2(G) \neq \gamma_2(G - v)$, then removing $v$ change the 2-domination number, which implies
	$st_{\gamma_2}(G) = 1,$
	and so the inequality holds trivially. \qed
	\end{proof} 
	
	By applying this proposition iteratively, for vertices $v_1, v_2, \ldots, v_s$ with $1 \leq s \leq n-2$ and $n = |V(G)|$, we get
	\[
	st_{\gamma_2}(G) \leq st_{\gamma_2}(G - v_1 - \cdots - v_s) + s.
	\]
	Using this formula, various upper bounds for $st_{\gamma_2}(G)$ can be obtained.
	In the next theorem, we state some of these upper bounds. The proof for each case involves removing vertices from $G$ until the induced subgraph satisfying the hypothesis appears. Then, applying the above inequality and the known value of $st_{\gamma_2}(G - v_1 - \cdots - v_s)$ yields the result.

	\begin{theorem}
		Let $G$ be a simple graph of order $n \geq 2$. Then:
		\begin{enumerate}
			\item[(i)] $st_{\gamma_2}(G) \leq n - 1$.
			\item[(ii)] If $G$ has the star graph $K_{1,t}$ as the induced subgraph with $t \geq 3$, then $st_{\gamma_2}(G) \leq n - t $.
		%	\item[(iii)] For any graph $G$, $st_{\gamma_2}(G) \leq n - \Delta(G)$.
		\end{enumerate}
	\end{theorem}

	We need the following theorem: 
	
	\begin{theorem}{\rm\cite{DAM}}\label{upper}
		\begin{enumerate}
			\item[(i)]	
	If the minimum degree $\delta(G)$ is $0$ or $1$, then $\gamma_2(G)$ can be equal to $n$.
	\item[(ii)] 
	 If $\delta(G)=2$, then $\gamma_2(G)\leq \frac{2}{3}n$. 
	 \item[(iii)] 
	  If $\delta(G)\geq 3$, then $\gamma_2(G)\leq \frac{1}{2}n$. 
	 	\end{enumerate}
	 \end{theorem}

Now we state and prove the following corollary.

	\begin{corollary}\label{T33}
		\begin{enumerate}
			\item[(i)]	
			If $G$ is a graph of order $n$ with $\delta(G)=2$, then 
		\[
		st_{\gamma_2}(G) \leq n + 1 -\frac{3\gamma_2(G)}{2}.
		\]		
		\item[(ii)]
		If $G$ is a graph of order $n$ with $\delta(G)\geq 3$, then 
		\[
		st_{\gamma_2}(G) \leq n + 1 -2\gamma_2(G).
		\]
		\end{enumerate} 
	\end{corollary}
	\begin{proof}
			\begin{enumerate}
				\item[(i)]
		Let $st_{\gamma_2}(G) = k$. By definition of stability, removal of any set $S = \{v_1, \ldots, v_{k-1}\}$ with $k-1$ vertices preserves the $2$-domination number, i.e.,  
		$\gamma_2(G) = \gamma_2(G - v_1) = \cdots = \gamma_2(G - v_1 - \cdots - v_{k-1}).$
		For the remaining graph $G - S$ of order $n - (k-1) = n - k + 1$, we use the known upper bound for $2$-domination in Part (i) of Theorem \ref{upper},
		$\gamma_2(G - S) \leq \left\lceil \frac{2(n - k + 1)}{3} \right\rceil.$
		Since $\gamma_2(G - S) = \gamma_2(G)$, we have
		$\gamma_2(G) \leq \frac{2(n - k + 1)}{3}.$
		Solving for $k$ yields:
		$3\gamma_2(G) \leq 2n - 2k + 2 $ which implies $k \leq n + 1 - \frac{3}{2}\gamma_2(G).$
		Therefore, 
		$	st_{\gamma_2}(G) \leq n + 1 -\frac{3}{2}\gamma_2(G).$
		
		\item[(ii)] With similar argument in the proof of Part (i) and 
		the upper bound for $2$-domination in Part (ii) of Theorem \ref{upper}, we have the result. \qed
		\end{enumerate} 
	\end{proof}
\begin{corollary}
	Let $G$ be a graph of order $n \geq 2$. If $st_{\gamma_2}(G) = n - 1$, then $\gamma_2(G) = 1$.
\end{corollary}
\begin{proof}
	It is a direct consequence of Theorem \ref{T33}.
\end{proof}
	
	\begin{theorem}\label{up}
		Let $G$ be a graph of order $n$ such that $\gamma_2(G) \geq 2$. Then
		$st_{\gamma_2}(G) \leq n - \gamma_2(G)+1.$
	\end{theorem}
	
	\begin{proof}
		Let $D$ be a minimum $2$-dominating set of $G$, with $|D| = \gamma_2(G)$. By the definition of $2$-dominating, every vertex in $V(G) \setminus D$ has at least two neighbors in $D$.
				By removing all vertices in $V(G)\setminus D$, the $2$-domination number does not change and so $st_{\gamma_2}(G) \leq n - \gamma_2(G)+1.$\qed
	\end{proof}

	\begin{theorem}
		Let $G$ be a graph of order $n \geq 2$ with $\gamma_2(G) \geq 2$ and $\gamma_2(\overline{G}) \geq 2$. Then
		\[
		st_{\gamma_2}(G) + st_{\gamma_2}(\overline{G}) \leq 2n-2.
		\]
		\end{theorem}
	
	\begin{proof}
		We have 	$\gamma_2(G) + \gamma_2(\overline{G}) \geq 4.$
		Using the stability bound in Theorem \ref{up}
		$st_{\gamma_2}(G) \leq n - \gamma_2(G)+1.$ It follows that:
		\begin{align*}
		st_{\gamma_2}(G) + st_{\gamma_2}(\overline{G}) &\leq (n - \gamma_2(G)+1) + (n - \gamma_2(\overline{G})+1) \\
		&= 2n+2- (\gamma_2(G) + \gamma_2(\overline{G})) \\
		&\leq 2n-2.
		\end{align*}
		\qed
	\end{proof}

\section{Results for some operations of two graphs}
In this section, we study the stability of 2-domination number of some operations of two graphs. First we consider the join of two graphs. 
The join $ G\vee H$ of two graphs $G$ and $H$ with disjoint vertex sets $V(G)$  and edge sets $E(G)$ is the graph union $G\cup H$ together with all the edges joining $V(G)$. 

\begin{theorem}\label{join2dom}
	If $G$ and $H$ are nonempty graphs, then
	\[
	\gamma_2(G \vee H) = \min\{\gamma_2(G), \gamma_2(H)\}.
	\]
\end{theorem}
\begin{proof} 
	Suppose that $|V(G)|=n_1$ and $|V(H)|=n_2$. It is clear to achieve that $\gamma_{2}(G\vee H)\geq 2$. Let $1\leq i \leq n_{1}+n_{2}$.
	We can see that for every $D_{1}\subseteq V(G)$ and $D_{2} \subseteq V(H)$ such that $\vert D_{1} \vert=i_{1}$ and $\vert D_{2} \vert=i_{2}$ where $i_{1}+i_{2}=i$, $D_{1} \cup D_{2}$ is a $2$-dominating set of $G\vee H$. Moreover, if $D$ is a $2$-dominating set for $G$ or $H$ of size $i$ then $D$ is the $2$-dominating set for $G \vee H$. Therefore we have the result. \qed
	\end{proof}

By Theorem \ref{join2dom}, we have the following result.

\begin{theorem}
	Let $G$ and $H$ be two nonempty graphs, then
	\[
	st_{\gamma_2}(G \vee H) \leq \min \left\{ st_{\gamma_2}(G), \, st_{\gamma_2}(H) \right\}.
	\]
\end{theorem}

Here, we recall the definition of lexicographic product of two graphs.  
For two graphs $G$ and $H$, let $G[H]$ be the graph with vertex
set $V(G)\times V(H)$ and such that vertex $(a,x)$ is adjacent to vertex $(b,y)$ if and only if
$a$ is adjacent to $b$ (in $G$) or $a=b$ and $x$ is adjacent to $y$ (in $H$). The graph $G[H]$ is the
lexicographic product (or composition) of $G$ and $H$, and can be thought of as the graph arising from $G$ and $H$ by substituting a copy of $H$ for every vertex of $G$ \cite{DAM1}.

The following theorem gives the $2$-domination number of $G[H]$.

\begin{theorem}\label{T44}
	If $G$ and $H$ are two nonempty graphs. Then 
	\[
 \gamma_2(G[H]) \leq |V(H)| \cdot \gamma_2(G).
	\]
\end{theorem}
\begin{proof}
	Let $D_G$ be a minimum 2-dominating set of $G$ of size $\gamma_2(G)$. Consider the lexicographic product $G[H]$. 
	Construct the set 
	$D = D_G \times V(H),$
	which includes all vertices in the copies of $H$ corresponding to vertices in $D_G$. Since $|D| = |V(H)| \cdot \gamma_2(G)$, we only need to check that $D$ is 2-dominating in $G[H]$.
	
	For any vertex $(x,y) \in V(G[H])$, if $x \in D_G$ then $(x,y) \in D$. Otherwise, since $D_G$ is 2-dominating in $G$, $x$ has at least two neighbors in $D_G$, say $g_1$ and $g_2$. Then $(x,y)$ is adjacent to every vertex in the copies $H_{g_1}$ and $H_{g_2}$, which are subsets of $D$. Thus, $(x,y)$ has at least two neighbors in $D$.
	Hence, $D$ is a $2$-dominating set of $G[H]$ and
	\[
	\gamma_2(G[H]) \leq |D| = |V(H)| \cdot \gamma_2(G).
	\]\qed
\end{proof}

By Theorem \ref{T44}, we have the following result.
 	
 	\begin{corollary}
 		Let $G$ and $H$ be two nonempty graphs. Then
 		\[
 		st_{\gamma_2}(G[H]) =
 		\begin{cases}
 		st_{\gamma_2}(G), & \text{if } G \text{ has no isolated vertex}, \\
 		\min\{st_{\gamma_2}(G), st_{\gamma_2}(H)\}, & \text{if } G \text{ has at least one isolated vertex}.
 		\end{cases}
 		\]
 	\end{corollary}

Now, we obtain the stability of $2$-domination number of corona of two graphs. We first state and prove the following theorem.

\begin{theorem}{\rm \cite{M1}}\label{corona1}
Suppose that $G$ is a graph of order $n$ and $H$ is any graph with no universal vertex. Then, $\gamma_2(G\circ H) =|V(G)|+\gamma_2(H)$.
\end{theorem}

\begin{proof}
	Take the vertex set $V(G)$ together with a minimum $2$-dominating set of $H$ in one copy. This forms a $2$-dominating set for $G \circ H$. Obviously this set is a $2$-dominating set with minimum size. So we have the result.\qed	
\end{proof}

\begin{remark}
	 If $H$ contains a universal vertex (i.e., a vertex adjacent to all others in $H$), then $\gamma_2(G \circ H) = n$.
\end{remark}

By Theorem \ref{corona1}, we have the following result.

\begin{corollary}
	If  $G$ and $H$ are two  graphs, then $st_{\gamma_2}(G\circ H)=1 $. 
\end{corollary} 

\begin{proof} 	By Theorem \ref{corona1}, $\gamma_2(G\circ H) =|V(G)|+\gamma_2(H)$, so removing any vertex from the corona product either disconnects a root vertex in $G$ or breaks the $2$-dominating set in a copy of $H$, thus changing $\gamma_2(G \circ H)$. Hence, we have the result. \qed
\end{proof}

\section{Conclusion}
	This paper introduces the concept of the stability of $2$-domination number of graphs and explores various properties related to its number. We have determined the precise values of stability of $2$-domination number for specific graphs. There is much work to be done in this area.
	\begin{enumerate}
		\item Define the edge stability of  $2$-domination number and study its properties. 
		\item What is the stability of $2$-domination number of  natural and fractional powers of a graph?

		\item Study the complexity of the stability of $2$-domination number  for many of the graphs.  
		
	\end{enumerate}

\end{document}